\documentclass[hidelinks,onefignum,onetabnum]{siamart251216}

\usepackage{enumerate}
\usepackage{amsfonts}
\usepackage{graphicx}
\usepackage{subcaption} 
\usepackage{epstopdf}
\usepackage[dvipsnames]{xcolor}

\usepackage{algorithmic}
\ifpdf
  \DeclareGraphicsExtensions{.eps,.pdf,.png,.jpg}
\else
  \DeclareGraphicsExtensions{.eps}
\fi

\newcommand{\N}{\mathbb N}

\newcommand{\R}{\mathbb R}
\newcommand{\C}{\mathbb C}

\newcommand{\St}{\operatorname{St}}

\newcommand{\mcR}{\mathcal R}

\newcommand{\tr}{\operatorname{tr}}
\newcommand{\sym}{\operatorname{sym}}
\newcommand{\Skew}{\operatorname{skew}} 

\newcommand{\dist}{\operatorname{dist}}
\newcommand{\Exp}{\operatorname{Exp}}
\newcommand{\Log}{\operatorname{Log}}
\newcommand{\Cay}{\operatorname{Cay}}

\DeclareMathOperator{\diag}{diag}
\makeatletter
\renewcommand*\env@matrix[1][*\c@MaxMatrixCols c]{%
  \hskip -\arraycolsep
  \let\@ifnextchar\new@ifnextchar
  \array{#1}}
\makeatother

\newsiamremark{remark}{Remark}
\newsiamremark{hypothesis}{Hypothesis}
\crefname{hypothesis}{Hypothesis}{Hypotheses}
\newsiamthm{claim}{Claim}
\newsiamremark{fact}{Fact}
\crefname{fact}{Fact}{Facts}

\headers{Error bounds for second-order Stiefel retractions}{R. Jensen, and R. Zimmermann}

\title{Conditioning and interpolation error bounds for second-order Stiefel retractions with closed-form inverses \thanks{Submitted to the editors DATE.
\funding{This work was supported by the Independent Research Foundation Denmark, DFF, grant nr. 3103-00094B}}}

\author{Rasmus Jensen\thanks{Department of  Mathematics and Computer Science, SDU Odense 
  (\email{rasmusj@imada.sdu.dk}, \email{zimmermann@imada.sdu.dk}).}
\and Ralf Zimmermann\footnotemark[2]}

\usepackage{amsopn}
\usepackage{physics}

\usepackage[normalem]{ulem}

\ifpdf
\hypersetup{
  pdftitle={Conditioning and interpolation error bounds for second-order Stiefel retractions with closed-form inverses},
  pdfauthor={R. Jensen, and R. Zimmermann}
}
\fi

\begin{document}

\maketitle

\begin{abstract}
Retractions provide a computationally efficient alternative to the Riemannian exponential and logarithm maps for practical data-processing tasks on manifolds. 
In particular, second-order retractions with closed-form inverse are well-suited for interpolation problems on manifolds.
On the Stiefel manifold of orthogonal frames, there are only two retractions of this type: the Cayley retraction, which is second-order accurate under the canonical metric, and the recently proposed polar-light retraction, which is second-order accurate under the Euclidean metric.

In this paper, we study the properties of these maps in the context of interpolation on the Stiefel manifold. 
To obtain explicit interpolation error bounds, we examine the conditioning of the retraction maps and their inverses. We show that the retractions are well-conditioned, and we derive interpolation error bounds similar to those of classical Euclidean interpolation. The inverse retractions are not well-conditioned in general, and we discuss how data can be mapped via an isometric group action to ensure stable computations. 
As with all retractions on compact manifolds, the inverse canonical Cayley retraction and the inverse polar-light retraction exist only locally, and we construct normal neighborhoods around any point in which either the inverse Cayley retraction or the invese polar-light retraction are guaranteed to be computable. 

As an application of the retraction maps, we consider Hermite interpolation, where the objective is to reproduce both sampled function values and derivative information. A numerical example demonstrates that retraction-based interpolation is competitive with classical methods based on Riemannian normal coordinates.

\end{abstract}



\section{Introduction} 

Interpolation of data on the Stiefel manifold of column-\\ orthogonal matrices $\St(n,p)=\{U\in \R^{n\times p}: U^TU=I_p\}$ has received considerable attention due to its applications in parametric model-order reduction \cite{ ElOmari:2025,Friderikos2021,Zimmermann2018}. The task, given sample data $U_1,\dots,U_m$ at time instances $t_1<\dots<t_m$, is to construct a curve $c:I\to \St(n,p)$ so that $c(t_i)=U_i$, for all $i$. 

Interpolation on the Stiefel manifold, as well as on general Riemannian manifolds, is classically performed using Riemannian normal coordinates. Choosing an anchor $\hat U$, which may, for example, be one of the data points or their Riemannian center of mass, one maps all data to the tangent vector space $T_{\hat U}\St(n,p)$, applies an interpolation scheme (Lagrange, Hermite, splines,\dots) there, and maps the computed interpolant back to the manifold. The advantage of this approach is that the user can choose any off-the-shelf method from classical Euclidean interpolation in vector spaces. 

Riemannian normal coordinates are favorable for establishing interpolation error bounds, because they allow for distance estimates based on the manifold's curvature. Using local sectional curvature information, in \cite{ZimmermannHermite_2020}, it was shown that under Riemannian normal coordinates, employing an interpolation scheme of order $k$ in the tangent space yields a manifold interpolant of the same asymptotic order $\mathcal{O}(h^k)$, as $h\to 0$.\footnote{Strictly speaking, this result is an immediate consequence of \cite[Thm. 3]{ZimmermannHermite_2020}, if the anchor point is not one of the sample points.} A related result based on global curvature bounds is in \cite{Jacobsson:2024}. When replacing the Riemannian exponential map with a retraction, the work \cite{SeguinKressner:2024} proves that Hermite interpolation via a Riemannian De Casteljau algorithm exhibits an asymptotic error $\mathcal{O}(h^4)$ as $h\to 0$. The paper \cite{jensen2025maxvol} shows that for the Grassmann manifold and a particular choice of local coordinates, Euclidean interpolation error estimates transfer directly to the manifold, preserving their order.

\textit{Original contribution.} We derive explicit interpolation error bounds for interpolation schemes on the Stiefel manifold under either the polar-light retraction or the canonical Cayley retraction, and their respective inverses. For the retractions at hand, the process of mapping data from the tangent space to the manifold is well-conditioned, while mapping data to the tangent space can be arbitrarily ill-conditioned. As the Stiefel manifold is homogeneous, we employ a data-centering scheme that sends a given sample data cloud isometrically to a local domain with favorable conditions for numerical computations. We furthermore construct two respective explicit normal neighborhoods around any point $\hat U\in \St(n,p)$, for which we are guaranteed that the inverse Cayley retraction and the inverse polar-light retraction, respectively, is well-defined. 

Secondly, we consider Hermite interpolation on the Stiefel manifold. 
In contrast to \cite{ZimmermannHermite_2020}, where the use of the Riemannian normal coordinates made finite-difference approximation unavoidable for mapping derivative information to the tangent space, here, we map derivatives using the differential of the inverse retractions.
For the polar-light retraction in its original form, the Fr\'echet derivative of the matrix logarithm appears.
To avoid the costly computation, which would require applying an inverse scaling and squaring method \cite{Mohy:2013} or numerical integration, we instead replace the matrix exponential and logarithm with the Cayley transformation and its inverse. Computing the resulting differential requires solving a ($p\times p$) Lyapunov equation whose solution exists and is stable whenever the inverse polar-light retraction with the Cayley transformation replacing the matrix logarithm is stable.

\textit{Organization:} In \Cref{sec:background} we recall the Stiefel manifold and the polar-light coordinate chart of \cite{Jensen:2026}, and gather known conditioning results for the matrix polar decomposition. \Cref{sec:err_analysis} is the main section of this paper where we carry out the error analysis and derive manifold interpolation error bounds. In \Cref{sec:Hermite_interpol} we discuss Hermite interpolation using the retractions of interest and carry out a numerical experiment. The paper is concluded in \Cref{sec:concluding_remarks}.

\section{Background}\label{sec:background}

We set the scene by recalling the basic properties of the Stiefel manifold, supplemented by the recently introduced coordinate chart of \cite{Jensen:2026}. For an elaborate discussion of the Stiefel manifold we refer the reader to \cite{AbsilMahonySepulchre2008,EdelmanAriasSmith1999}. We also quickly discuss the conditioning of the factors in the matrix polar decompositions from \cite{Higham:2008:FM}. 

A few matrix preliminaries are due.
For $p\in \N$, the identity matrix is denoted by $I_p\in\R^{p\times p}$, or simply $I$, if the dimension is clear.
The $(p\times p)$-orthogonal group is denoted by
$
  O(p) = \{Q \in \R^{p\times p}\mid Q^TQ = QQ^T = I_p\}.
$
The special orthogonal group is $SO(p)= \{Q\in O(p)\mid  \det(Q) = 1\}$.
The sets of symmetric and skew-symmetric $(p\times p)$-matrices are $\sym(p) = \{A\in\R^{p\times p}|A^T=A\}$ and
$\Skew(p) = \{A\in\R^{p\times p}|A^T=-A\}$, respectively.
The set of $(p\times p)$-symmetric positive definite matrices is denoted by $SPD(p)$.
The matrix exponential and principal matrix logarithm are denoted by
\[
 \exp_m(X):=\sum_{j=0}^\infty{\frac{X^j}{j!}}, \quad \log_m(I+X):=\sum_{j=1}^\infty{(-1)^{j+1}\frac{X^j}{j}}.
\]
It holds $\exp_m\big\vert_{\Skew(p)}: \Skew(p)\to SO(p)$, $\log_m\big\vert_{SO(p)}: SO(p) \to \Skew(p)$ when well-defined.

The Cayley transformation and its inverse can be used as structure-preserving approximations of the matrix exponential and logarithm. The definitions read
\begin{subequations}
\begin{align}
\label{eq:Cay_trafo}
    \Cay(A)&=&\qty(I-A)^{-1}\qty(I+A), 
    \, \quad & \Cay\big\vert_{\Skew(p)}:& \Skew(p)\to SO(p) \\
\label{eq:Cay_trafo_inv}
    \Cay^{-1}(R)&=&(R+I)^{-1}(R-I),
    \, \quad & \Cay^{-1}\big\vert_{SO(p)}:& SO(p)\to \Skew(p). \ 
\end{align}
\end{subequations}
It holds that $\exp_m(A)\approx \Cay(\frac{1}{2}A)$ and $\log_m(R)\approx 2\Cay^{-1}(R)$ and the approximations are accurate up to a Taylor expansion of second order.

\subsection{The Stiefel manifold}\label{sec:background_Stiefel} The Stiefel manifold of orthogonal $p$-frames is  
\begin{equation*}
    \St(n,p)=\qty{U\in \R^{n\times p}:U^TU=I_p}.
\end{equation*}
It is a compact smooth manifold of dimension $\frac{1}{2}p(p-1)+(n-p)p$ and is Riemannian when endowing the tangent space at each anchor point $\hat U\in St(n,p)$,
\begin{equation*}
    T_{\hat U}\St(n,p)=\qty{\xi\in \R^{n\times p}:\hat U^T\xi+\xi^T\hat U=0},
\end{equation*}
with an inner product, e.~g., 
with the Euclidean metric $g_e(\xi, \Delta)=\tr(\xi^T\Delta)$ or the canonical metric $g_c(\xi, \Delta)=\tr(\xi^T(I_n-\tfrac{1}{2}UU^T)\Delta)$. A general parametric family of $\beta$-metrics is discussed in \cite{HueperMarkinaLeite2020}, for which the Riemannian exponential map is given by 
\begin{equation}
\label{eq:Riemann_exp}
    \Exp_{\hat U}(\xi)=\begin{bmatrix}
        \hat U&\hat U_\perp
    \end{bmatrix}\exp_\mathrm{m}\left(\begin{bmatrix}
        2\beta A&-B^T\\
        B&0
    \end{bmatrix}\right)   
    \begin{bmatrix}
   I_p\\
    0
  \end{bmatrix}\exp_{\mathrm{m}}((1-2\beta)A), \hspace{0.2cm} \xi = \hat UA + \hat U_\perp B.
\end{equation}
Here, $\hat U_\perp$ is an orthogonal completion such that $\begin{bmatrix}
        \hat U&\hat U_\perp
    \end{bmatrix}\in O(n)$.
The Euclidean metric is obtained from $\beta=1$, and  $\beta=\frac{1}{2}$ yields the canonical metric. 

\subsection{The polar light retraction}
\label{sec:polar_light_background}

We now state the local coordinate charts termed polar-light coordinates in \cite{Jensen:2026}. For each fixed $\hat U \in \St(n,p)$, the map
\begin{equation}\label{eq:pl_retraction}
   T_{\hat U}\St(n,p)\ni \xi \mapsto \varphi_{\hat U}(\xi) = \left(\hat U\exp_m(\hat U^T\xi) + (I-\hat U\hat U^T)\xi\right)\left(I_p+\xi^T(I-\hat U\hat U^T)\xi\right)^{-\frac12}
\end{equation}
defined on a domain around $0\in T_{\hat U}\St(n,p)$ is a second-order retraction under the Euclidean metric. The main difference to the standard polar factor retraction \cite[eq. (4.7)]{AbsilMahonySepulchre2008} is the appearance of the $(p\times p)$ matrix exponential. It enables computing a closed-form inverse 
\begin{equation}\label{eq:pl_inverse}
    \St(n,p)\ni U \mapsto \psi_{\hat U}(U) = \hat U\log_m\left( \hat U^TU (U^T \hat U \hat U^TU)^{-\frac12}\right) + (I- \hat U\hat U^T)U(U^T\hat U \hat U^TU)^{-\frac12}.
\end{equation}
Note that $\hat U^TU (U^T \hat U \hat U^TU)^{-\frac12}\in O(p)$ is the orthogonal polar factor of $\hat U^TU$, see \Cref{sec:cond_matrix_polar}.

Let $E=\begin{bmatrix}
    I_p\\
    0
\end{bmatrix}$. 
Due to the homogeneous space structure of $\St(n,p)$, we can move from any point to any point via a group action.
To make this explicit, consider 
$\hat Q=\begin{bmatrix}
    \hat U&\hat U_\perp
\end{bmatrix}\in O(n)$ and observe that $\hat Q$ moves $\hat U$ to $E$ via $\hat Q^T \hat U = E$, and any tangent vector 
$\xi=\hat UA+\hat U_{\bot} B \in T_{\hat U}St(n,p)$ is mapped to its block coordinates $A,B$ via $\hat Q^T\xi = \hat Q^T(UA+\hat U_{\bot} B) =\begin{bmatrix}
    A\\ B
\end{bmatrix} $.
Then
\begin{equation}\label{eq:psi_Q}
    \psi_{\hat U}(U) = \hat Q \psi_{E}(\hat{Q}^TU),\quad 
    \psi_E\qty(\begin{bmatrix}
        U_1\\ U_2
    \end{bmatrix}) =  \begin{bmatrix}
        \log_m(U_1 (U_1^TU_1)^{-\frac12})\\ 
        U_2(U_1^TU_1)^{-\frac12})
    \end{bmatrix}.
\end{equation}
The corresponding parametrization is 
\begin{equation}\label{eq:varphi_Q}
    \varphi_{\hat U}(\xi) = \hat Q \varphi_E(\hat Q^T \xi),
    \quad 
    \varphi_E\qty(\begin{bmatrix}
        A\\B
    \end{bmatrix}) =  \begin{bmatrix}
        \exp_m(A)\\ B
    \end{bmatrix} (I_p + B^TB)^{-\frac12}.
\end{equation}
For our analysis in \Cref{sec:err_analysis} and the practical example provided in \Cref{sec:practical_example}, we will not explicitly use the retraction \eqref{eq:pl_retraction} and its inverse \eqref{eq:pl_inverse}, but instead rely on \eqref{eq:psi_Q} and \eqref{eq:varphi_Q}. 
The data transformation with $\hat Q$ is considered as a preprocessing and postprocessing step that can be executed efficiently by representing $\hat Q$ as a low-rank modification of the identity $I_n$ as in \cite{Jensen:2026}.


\subsection{The Cayley retraction}
The Cayley retraction is associated with the canonical metric and stems from replacing the matrix exponential in \eqref{eq:Riemann_exp} with the Cayley transformation \eqref{eq:Cay_trafo}.
This yields \cite{Wen:2013} 
\begin{equation}\label{eq:cayley_retraction}
    \mcR_{\hat U}(\xi)=\Cay\qty(\tfrac{1}{2}(P_{\hat U}\xi {\hat U}^T-{\hat U}\xi^TP_{\hat U})){\hat U},
\end{equation}
where $P_{\hat U}=I_n-\frac{1}{2}\hat U\hat U^T$. It is a second-order retraction under the canonical metric. Similar to the Cayley retraction on the symplectic Stiefel manifold \cite{BendokatZimmermann:2021}, and which was also noted in \cite{Tiep:2025}, we can evaluate \eqref{eq:cayley_retraction} as follows
\begin{equation*}
    \mcR_{\hat U}(\xi)=-\hat U+(\hat U_\perp B+2\hat U)\qty(\frac{1}{4} B^TB-\frac{1}{2}A+I_p)^{-1},
    \quad \xi=\hat UA+\hat U_\perp B,
\end{equation*}
which only requires inverting a ($p\times p$) matrix. The Cayley retraction has a closed-form inverse $\mcR^{-1}_U:\St(n,p)\to T_{\hat U}\St(n,p)$
\begin{equation}\label{eq:inv_Cay_ret}
    \mcR_{\hat U}^{-1}(U)=2{\hat U}F(U)^T+2UF(U)-2{\hat U}, \quad F(U)=(I+{\hat U}^TU)^{-1}. 
\end{equation}
If $\hat U=E,\xi=\begin{bmatrix}
    A\\
    B
\end{bmatrix}$ and $U=\begin{bmatrix}
    U_1\\
    U_2
\end{bmatrix}$, then
\begin{align*}
   \mcR_{E}(\xi)&=\Cay\qty(\frac{1}{2}\begin{bmatrix}
       A&-B^T\\
       B&0
   \end{bmatrix})E, \\ 
   \mcR^{-1}_E(U)&=2\begin{bmatrix}
       (I_p+U_1)^{-T} - (I_p+U_1)^{-1}\\
       U_2(I_p+U_1)^{-1}
   \end{bmatrix}.
\end{align*}

The matrix $(I_p+U_1)^{-T} - (I_p+U_1)^{-1}$ is skew-symmetric. If $\eta\in T_{\hat U}\St(n,p)$ and $\hat Q=\begin{bmatrix}
    \hat U&\hat U_\perp
\end{bmatrix}$, then $R_{\hat U}(\eta)=\hat Q\mcR_E(\hat Q^T\eta)$ and $\mcR^{-1}_{\hat U}(U)=\hat Q\mcR_{E}^{-1}(\hat Q^T U)$, which, similar to the formulas for the polar-light retraction in \eqref{eq:varphi_Q}, provides an explicit relationship between points on $\St(n,p)$ and local coordinates in $\Skew(p)\times \R^{(n-p)\times p}$, computed via the Cayley retraction. 

To the best of our knowledge, the polar-light retraction and the canonical Cayley retraction are the only existing second-order retractions for the Stiefel manifold that admit an inverse in closed form.

\subsection{Conditioning and the matrix polar decomposition}
\label{sec:cond_matrix_polar}
For a differentiable map between vector spaces $F:X\to Y$, the absolute condition number (or simply the conditioning) of $F$ at $x\in X$ is the operator norm
\begin{equation}\label{eq:def_abs_cond}
    \|dF_x\|=\sup_{v\in X}\frac{\|dF_x[v]\|_{Y}}{\|v\|_{X}},
\end{equation}
where $\|\cdot \|_{X},\|\cdot \|_{Y}$ are norms defined on $X$ and $Y$ respectively \cite[p. 56]
{Higham:2008:FM}. 

Let $A\in \R^{m\times m}$ be nonsingular. Then it has a unique polar decomposition $A=UH$ with $U=A(A^TA)^{-\frac12}\in SO(m)$ and $H=(A^TA)^{\frac{1}{2}}\in SPD(m)$ \cite[Theorem 8.1]{Higham:2008:FM}. If we let $A+ \Delta A=\tilde U\tilde H$ be a perturbation of $A=UH$, then 
\begin{align}
    \|H-\tilde H\|_F&\leq \sqrt{2}\|A-\tilde A\|_F\label{eq:dH_boubd},\\
    \|U-\tilde U\|&\leq \frac{2}{\sigma_m+\tilde \sigma_m}\|A-\tilde A\| \label{eq:UUtilde_bound},
\end{align}
where $\sigma_m$ and $\tilde \sigma_m$ are the smallest singular values of $A$ and $\tilde A$, respectively, and $\|\cdot\|$ is any unitarily invariant norm \cite[Theorems 8.9 and 8.10]{Higham:2008:FM}. By the mean-value theorem and the definition of the condition number, computing $H$ is seen to be well-conditioned, with its condition number bounded by $\sqrt{2}$. Computing the $U$-factor becomes ill-conditioned as $A$ approaches rank-deficiency.

\section{Error analysis for data processing and interpolation error bounds}\label{sec:err_analysis}

Given data $U^{(i)}=F(t_i)$ from a function $F:I\to \St(n,p), t\mapsto F(t)$ sampled at $t_1<t_2<\dots <t_m$, we consider the construction of a manifold interpolant $\tilde F:I\to \St(n,p)$ so that $\tilde F(t_i)=F(t_i)$, for all $i$. The generic process is as follows:
\begin{enumerate}
    \item  Map the manifold data to their local coordinate images using a fixed coordinate chart. The coordinate images are situated in a vector space.
    \item Interpolate the local coordinate images using any Euclidean interpolation method. 
    \item Map back to the manifold using the parameterization corresponding to the chosen coordinate chart. 
\end{enumerate}
In this section, we first consider errors associated with using the polar-light retraction \eqref{eq:pl_retraction} and its inverse \eqref{eq:pl_inverse}. 
While retractions are defined as maps between the tangent space $T_{\hat U}St(n,p)$ and $\St(n,p)$, they can always be considered as parameterizations on a Euclidean coordinate domain,
because each tangent space is isomorphic to the Euclidean $\R^d$ of the same dimension. For the Stiefel manifold, a tangent vector $\xi = \hat U A + \hat U_\bot B$ has Euclidean coordinates
$\begin{bmatrix}
    A\\
    B
\end{bmatrix}\in \Skew(p) \times \R^{(n-p)\times p}.$
Our analysis will be based on the interpretation of the retractions as maps from coordinate domain to manifold, i.~e., as local parameterizations.

The overall goal of the section is to derive interpolation error bounds of the form
\begin{equation*}
    \dist(F(t),\tilde F(t))\leq Kh^k,
\end{equation*}
with explicit constant $K$. Here $\tilde F:I\to \St(n,p)$ is an interpolant of $F$ obtained from Steps 1--3 above, and $h=\max_i |t_{i+1}-t_i|$ is the maximal step size. $\tilde F$ depends on the choice of parameterization, and so does the constant $K$.

\subsection{Mapping to polar-light coordinates}\label{sec:to_loc_coords}

We will consider the process of mapping manifold data to their polar-light coordinate images, facilitated by the map \eqref{eq:pl_inverse}, \eqref{eq:psi_Q}.

In order to compute the local coordinate matrix of a point $U=\begin{bmatrix}
    U_1\\
    U_2
\end{bmatrix}\in \St(n,p)$, we compute the polar decomposition of the upper ($p\times p$) block $U_1=RH, R\in O(p), H\in SPD(p)$. If $U_1$ is rank-deficient, the polar factor $R=U_1(U_1^TU_1)^{-\frac12}$ is not unique, $H=(U_1^TU_1)^{\frac{1}{2}}$ is positive semidefinite, and so the coordinate chart $\psi_E$ is not well-defined. It follows from \eqref{eq:UUtilde_bound} that the stability of computing the polar factor $R$ is governed by the smallest singular value of $U_1$, which may even be rank-deficient. 
While there is no way to improve the formal condition number, one can always apply a permutation matrix $P$ to $U$ such that $PU=\hat U$ has a non-singular, possibly even well-conditioned upper ($p\times p$) block. The matrix $P$ can be computed using a maximum-volume scheme \cite{Goreinov:1997}. In this work, however, we will use the block-Householder QR decomposition as in \cite{jensen2025maxvol} (see the upcoming \Cref{def:Phi}).

Computing the factor $H\in SPD(p)$ in the polar decomposition has absolute condition $\|d H_{U_1}\|_F\leq \sqrt{2}$ \cite[Theorem 8.8]{Higham:2008:FM}, leading to the following Lemma.
\begin{lemma}\label{lem:condition_invSPD_factor}
    Let $U=\begin{bmatrix}
        U_1\\
        U_2
    \end{bmatrix}$ with $U_1$ regular. Then the map $J(U_1)=(U_1^TU_1)^{-\frac12}$ has the condition 
    \begin{equation}\label{eq:condition_invSPD_factor}
        \|dJ_{U_1}\|_F\leq \sqrt{2}\frac{1}{\sigma_p^2},
    \end{equation}
    where $\sigma_p$ is the smallest singular value of $U_1$.
\end{lemma}
\begin{proof}
    Define $H(U_1) = (U_1^TU_1)^{\frac12}$ so that $J(U_1) = (H(U_1))^{-1}$.
    Using the chain rule and the general inequality  $\|ABC\|_F\leq \|A\|_2\|B\|_F\|C\|_2$ \cite[Corollary 3.5.10]{horn1991matrix}, we obtain for the directional derivative in direction $C$ with $\|C\|_F=1$
    \begin{align*}
        \|d J_{U_1}[C]\|_F&=\|((U_1^TU_1))^{-\frac{1}{2}}dH_{U_1}[C](U_1^TU_1)^{-\frac{1}{2}}\|_F\\
        &\leq \sqrt{2}\|(U_1^TU_1)^{-\frac{1}{2}}\|_2^2
        =\sqrt{2}\frac{1}{\sigma_p^2}.
    \end{align*}
    
\end{proof}

To obtain the skew-symmetric factor $A$ in \eqref{eq:psi_Q}, we must compute the (principal) matrix logarithm of the polar factor $R=U_1(U_1^TU_1)^{-\frac{1}{2}}$, which requires that $R$ does not have $-1$ as eigenvalue. Under the Frobenious-norm, the absolute condition of the matrix logarithm is 
\begin{equation*}
    \|d(\log_m)_M\|_F=\max\left\{1,\max_{\stackrel{\lambda,\mu\in \Lambda(M)}{\lambda\neq \mu}}\frac{|\log(\lambda)-\log(\mu)|}{|\lambda-\mu|}\right\},
\end{equation*}
where $\Lambda(M)$ is the spectrum of $M$ \cite[Equation 11.11]{Higham:2008:FM}.
In other words, the matrix logarithm is well-conditioned, if the polar factor $R$ does not feature a pair of eigenvalues which are close to, but situated on opposite sides of, the negative real axis in the complex plane \cite[pp. 273]{Higham:2008:FM}. 

To keep the bound \eqref{eq:condition_invSPD_factor} small, it is advantageous to move Stiefel data sets into neighborhoods of the canonical point $E$, and use $E$ as center for the local coordinate chart. To this end, we use the following group action, which maps a neighborhood of a designated point $\hat U\in \St(n,p)$ to a neighborhood around $E$.
\begin{definition}\label{def:Phi}
    Fix a point $\hat U\in \St(n,p)$. Construct a matrix $Q\in O(n)$ so that $Q^T\hat U=\begin{bmatrix}
        \diag(\pm 1,\dots,\pm 1)\\
        0
    \end{bmatrix}$, and take $S=E^TQ^T\hat U$. Then $\Phi_{(Q,S)}:\St(n,p)\to \St(n,p), U\mapsto Q^TUS$ is a bijection so that $\Phi_{(Q,S)}(\hat U)=E$. 
\end{definition}
In order to construct the matrix $Q$ in \Cref{def:Phi} we can proceed by applying the Householder QR decomposition to $\hat U$. This yields $R=Q^T\hat U$, where the upper ($p\times p$) block of $R$ is $\diag(\pm 1,\dots,\pm 1)$. 
Using the block representation of the Householder QR, the mapping can be evaluated in $\mathcal{O}(np^2)$ FLOPS. For details, see \cite[p. 12]{jensen2025maxvol}. 

The following result is immediate.
\begin{proposition}\label{prop:isometry}
    Endow $\St(n,p)$ with the Euclidean metric. Then $\Phi_{(Q,S)}:\St(n,p)\to \St(n,p)$ is a local isometry.
\end{proposition}
\begin{proof}
    By the definition of an isometry \cite[p. 12]{Lee2018riemannian}, the result follows from \cite[Proposition 2.51]{Lee2018riemannian}.
\end{proof}
Given data $U^{(1)},\dots, U^{(m)}$ in a neighborhood of $\hat U$, we can map them bijectively via $\Phi_{(Q,S)}$ from \Cref{def:Phi} to a neighborhood of $E$, so that pairwise distances are preserved (under the Euclidean metric). 

If the neighborhood is sufficiently small, we can also guarantee that the polar factor associated with the upper ($p\times p$)-block of each point has no eigenvalue equal to $-1$, and the smallest singular value of each upper block is bounded away from zero. In \Cref{sec:well_defined_inverse} we construct an explicit normal neighborhood of any point, in which \eqref{eq:psi_Q} is guaranteed to be well-defined.

\subsection{From polar-light coordinates to the manifold} 
\label{sec:to_manifold}
Given local coordinate matrices
$X^{(1)},\dots,X^{(m)}\in\Skew(p)\times \R^{(n-p)\times p}$ lying in a neighborhood of $0$ so that $\psi_E(U^{(j)})=X^{(j)}$, we now consider how distances are propagated when the $X^{(j)}$ are mapped  back to $\St(n,p)$ via \eqref{eq:varphi_Q}. For a linear map $L:\mathcal{V}\to \mathcal{W}$, where $\mathcal{V}$ and $\mathcal{W}$ are both matrix vector spaces, we consider the induced Frobenius operator norm
\begin{equation}\label{eq:ind_F_operator_norm}
    \|L\|_F=\sup_{M\neq 0}\frac{\|LM\|_F}{\|M\|_F}=\max_{\|M\|_F=1}\|LM\|_F,
\end{equation}
which is unitarily invariant. For any unitarily invariant matrix norm, it holds that $\|ABC\|\leq \|A\|_2\|B\|\|C\|_2$ \cite[Corollary 3.5.10]{horn1991matrix}, and any of the two norms on the right-hand side can be the 2-norm. Below, we collect some useful results needed for our analysis. 
\begin{lemma}\label{lem:prelim_res}
    Assume that $(n-p)\geq p$.
    Let $A\in\Skew(p)$ and $B\in \R^{(n-p)\times p}$. Define $S(B)=(I+B^TB)^{-\frac{1}{2}}$ and $G(B)=(I+B^TB)^{\frac{1}{2}}$, and let $B=\Gamma \begin{bmatrix} \Sigma\\ 0\end{bmatrix} V^T$ be a full singular value decomposition with $\Gamma\in O(n-p), V\in O(p)$ and $\Sigma=\diag(\sigma_1, \ldots, \sigma_p)\in\R^{p\times p}$. Then
    \begin{enumerate}
        \item $\|dG_B\|_F\leq 
        \max\left\{
    \frac{\sigma_1 }{\sqrt{1+\sigma_1^2}},
    \max_{i\neq j, i,j\leq p} \frac{\sqrt{2}\sqrt{\sigma_i^2 + \sigma_j^2}}{\sqrt{1+\sigma_i^2} + \sqrt{1+\sigma_j^2}}
    \right\} \leq \sqrt{2}$, 
        \item $\|S(B)\|_2=\frac{1}{\sqrt{1+\sigma_p^2}}\leq 1$,
        \item $\|d(\exp_m)_A\|_F\leq 1$.
    \end{enumerate}
\end{lemma}
\begin{proof}
We first address the bound on the differential of $G(B) = \sqrt{I+B^TB}$. With $f(B) = \begin{bmatrix}I\\ B\end{bmatrix}$ and $H(X) = \sqrt{X^TX}$, it holds $G(B) = H(f(B))$. Here, $H$ is again the mapping of matrix to its to symmetric polar factor for which we have the condition bound from 
\cite[Thm 8.8]{Higham:2008:FM}.
For $C$ with $\|C\|_F\neq 0$, 
 \begin{align*}
        \frac{\|DG_B[C]\|_F}{\|C\|_F} =
        \frac{\|DH_{f(B)}[Df_B[C]]\|_F}{\|C\|_F}\leq \| DH_{f(B)}\| _F\frac{\|Df_B[C]\|_F}{\|C\|_F}\\
        \stackrel{\text{\cite[Thm 8.8]{Higham:2008:FM}}}{\leq} \sqrt{2} \frac{\|\begin{bmatrix}0\\ C\end{bmatrix}\|}{\|C\|} = \sqrt{2}.
\end{align*}
In addition to the global bound, we also give an input-specific bound, which is more precise, when the singular values of $B$ are known.
Let $\Phi(B) = I+B^TB$ and $F(M) =\sqrt{M}$ so that $G(B) = (F\circ \Phi)(B)$.
Consider a direction matrix $C\in\R^{(n-p)\times p}$ with $\|C\|_F=1$.
The differential of $G$ at $B$ in direction $C$ is determined by the Sylvester equation
\[
    Y:=DG_B[C]\in \sym(p), \quad G(B)Y+YG(B) = C^TB + B^TC= D\Phi_B[C],
\]
see \cite[p. 134]{Higham:2008:FM}. Change coordinates according to the SVD $B=\Gamma  \begin{bmatrix} \Sigma\\ 0\end{bmatrix}  V^T$.
It holds $G(B) = V\sqrt{I+\Sigma^2}V^T$. Multiplying the Sylvester equation with $V^T$ from the left and $V$ from the right gives
\[
    \sqrt{I+\Sigma^2} V^TYV + V^TYV \sqrt{I+\Sigma^2} 
    = (V^TC^T\Gamma) \begin{bmatrix} \Sigma\\ 0 \end{bmatrix}+ \begin{bmatrix} \Sigma & 0 \end{bmatrix} (\Gamma^TCV).
\]
Introduce $X=V^TYV$, $D = \Gamma^TCV$ and note that $\|X\|_F=\|Y\|_F$, $\|D\|_F=\|C\|_F$.
Entry-wise, $X=X^T$ is specified by
\begin{align*}
    \sqrt{1+\sigma_i^2} X_{ij} + \sqrt{1+\sigma_j^2} X_{ji} = \sigma_j D_{ji} + \sigma_i D_{ij}
    \Leftrightarrow
    X_{ij} = \frac{\sigma_i D_{ij} + \sigma_j D_{ji} }{\sqrt{1+\sigma_i^2} + \sqrt{1+\sigma_j^2}}\\
    \Leftrightarrow
    \begin{cases}
    X_{ii} = \frac{\sigma_i }{\sqrt{1+\sigma_i^2}}D_{ii},\\
    \begin{pmatrix}
            X_{ij}\\
            X_{ji}
        \end{pmatrix} = \frac{1}{\sqrt{1+\sigma_i^2} + \sqrt{1+\sigma_j^2}}
        \begin{pmatrix}
            \sigma_i & \sigma_j\\
            \sigma_i & \sigma_j
        \end{pmatrix}
        \begin{pmatrix}
            D_{ij}\\
            D_{ji}
        \end{pmatrix}, \quad i\neq j\\
    \end{cases} \text{ for } 1\leq i,j\leq p.
\end{align*}
In the coordinates of the SVD of $B$, the differential is the linear map $D\mapsto X(D)$.
It holds $\frac{\|DG_B[C]\|_F}{\|C\|_F} = \frac{\|X(D)\|_F}{\|D\|_F}$.
Vectorizing $D$ in the ordering 
$$
\text{vec}^\Pi(D) = (D_{11}, D_{21}, D_{12},\ldots, D_{p1}, D_{1p},\ldots, ) \in \R^{(n-p)p}
$$
and writing $\delta_i = \frac{\sigma_i }{\sqrt{1+\sigma_i^2}}$, $\gamma_{ij}=\frac{1}{\sqrt{1+\sigma_i^2} + \sqrt{1+\sigma_j^2}}$, the linear map is realized by
\[
    \text{vec}^\Pi(D) \mapsto 
    \begin{pmatrix}[ccc|c]
        \delta_{1} &                                       &   &\\
                   & \gamma_{12}\begin{bmatrix}
                       \sigma_2 & \sigma_1\\
                       \sigma_2 & \sigma_1
                   \end{bmatrix}                           &   &  0\\
                   &                                       &\ddots &
    \end{pmatrix}
        \begin{pmatrix}
        D_{11}\\
        \begin{bmatrix}
                       D_{21}\\
                       D_{12}
        \end{bmatrix} \\
        \vdots\\
        \hline
        D_{p+1,1}\\
         \vdots
    \end{pmatrix}
    = \begin{pmatrix}
        X_{11}\\
        \begin{bmatrix}
                       X_{21}\\
                       X_{12}
        \end{bmatrix} \\
        \vdots
    \end{pmatrix} = \text{vec}^\Pi(X)\in \R^{p^2}.
\]
Let $\mathcal{X}\in\R^{p^2 \times (n-p)p}$ be the operator indicated in the above equation.
It holds $\frac{\|X(D)\|_F}{\|D\|_F}= \frac{\|\mathcal{X} (\text{vec}^\Pi(D))\|_2}{\|\text{vec}^\Pi(D)\|_2}\leq \|\mathcal{X}\|_2= \sqrt{\lambda_{\max}(\mathcal{X}\mathcal{X}^T)}$. Hence
\[
    \|\mathcal{X}\|_2 = \max\left\{
    \begin{array}{ll}
    |\delta_i|= \frac{\sigma_i }{\sqrt{1+\sigma_i^2}},& \quad i=1,\ldots,p\\
    \gamma_{ij}\sqrt{2}\sqrt{\sigma_i^2 + \sigma_j^2}
    = \frac{\sqrt{2}\sqrt{\sigma_i^2 + \sigma_j^2}}{\sqrt{1+\sigma_i^2} + \sqrt{1+\sigma_j^2}}
    , &\quad i\neq j, i,j\leq p
    \end{array}
    \right\}
\]
The diagonal term is increasing in $\sigma_i$ and thus maximal for $\frac{\sigma_1 }{\sqrt{1+\sigma_1^2}}\leq 1$. Because $\sqrt{\sigma_i^2 + \sigma_j^2}< \sqrt{1+\sigma_1^2}+\sqrt{1+\sigma_j^2}$, the factor associated with the off-diagonal terms is bounded by $\sqrt{2}$. In fact, the bound is approached if $\sigma_1\to \infty$ and $B$ is rank-deficient, i.e., $\sigma_p=0$. For $\|B\|_2\leq 1$, one can show that the bound is at most $\frac{1}{\sqrt{2}}$.

    The second claim follows from $\|(I+B^TB)^{-\frac12}\|_2=\|(I+\Sigma^2)^{-\frac{1}{2}}\|_2$. For the final claim take $\|\Delta\|_F=1$. A slight modification in the proof of \cite[Theorem 10.16]{Higham:2008:FM} yields
    \begin{align*}
    \|d(\exp_m)_A[\Delta]\|_F&=\norm{\int_0^1 \exp_m(A(1-s))\Delta \exp_m(As)\dd s }_F \\
    &\leq \|\Delta\|_F\int_0^1 \|\exp_m(A(1-s))\|_2\|\exp_m(As)\|_2\dd s=1,
\end{align*}
where we have used the unitary invariance of the Frobenius norm and that $A\in\Skew(p)$ so $1=\|\exp_m(At)\|_2,\forall t$. 
\end{proof}
\begin{lemma}\label{lem:cond_phi}
    The absolute condition of $\varphi_E$ at  $X= \begin{bmatrix}
        A\\B
    \end{bmatrix}\in \Skew(p)\times \R^{(n-p)\times p}$ in \eqref{eq:varphi_Q} is bounded by
    \begin{equation*}
        \|d (\varphi_E)_X\|_F\leq 3. 
    \end{equation*}
\end{lemma}
\begin{proof}
    Let $\Delta=\begin{bmatrix}
        \Delta A\\
        \Delta B
    \end{bmatrix}\in \Skew(p)\times \R^{(n-p)\times p}$ satisfy $\|\Delta\|_F=1$, and take $S(B)=(I+B^TB)^{-\frac{1}{2}}$ and $G(B)=(I+B^TB)^{\frac{1}{2}}$. Then by the product rule and the chain rule 
    \begin{equation}\label{eq:dvarphi}
        d(\varphi_E)_X[\Delta]=\begin{bmatrix}
        d(\exp_m)_A[\Delta A]\\
        \Delta B
    \end{bmatrix}S(B)+\begin{bmatrix}
        \exp_m(A)\\
        B
    \end{bmatrix}dS_B[\Delta B],
    \end{equation}
    where $dS_B[\Delta B]=-S(B)dG_B[\Delta B]S(B)$. 
    Taking norms, applying the triangle inequality, and using properties of the 2-norm and the Frobenius norm, we find
    \begin{align*}
        \|d(\varphi_E)_X[\Delta]\|_F
        \leq& \left\|\begin{bmatrix}
        d(\exp_m)_A[\Delta A]\\
        \Delta B
         \end{bmatrix} \right\|_F \|S(B)\|_2 \\
         & \quad +
        \left\|\begin{bmatrix}
        \exp_m(A) S(B)\\
        BS(B)
        \end{bmatrix} \right\|_2\|dG_B[\Delta B]\|_F\|S(B)\|_2
        \\
        \leq& \Biggl( \sqrt{\|d(\exp_m)_A(\Delta A)\|_F^2 +\|\Delta B\|_F^2}\\
        &+\sqrt{\|\exp_m(A)S(B)\|_2^2+\|BS(B)\|_2^2}\|dG_B[\Delta B]\|_F \Biggr)  \|S(B)\|_2.
    \end{align*}
    From \Cref{lem:prelim_res}: 
    \begin{itemize}
        \item $\|dG_B[\Delta B]\|_F= \|dG_B\|_F \|\Delta B\|_F \leq \sqrt{2}$.
        \item $\|S(B)\|_2=\frac{1}{\sqrt{1+\sigma_p^2}}\leq 1$.
        \item $\|d(\exp_m)_A[\Delta A]\|_F^2\leq \|\Delta A\|_F^2$.
    \end{itemize}
    Moreover, 
    $\|BS(B)\|_2=\max_i\qty{\frac{\sigma_i}{\sqrt{1+\sigma_i^2}}}= \frac{\sigma_1}{\sqrt{1+\sigma_1^2}}\leq 1$
    and $\|\exp_m(A)S(B)\|_2\leq 1$ as $\exp_m(A)\in SO(p)$.
    Hence, as $1=\|\Delta\|_F=\sqrt{\|\Delta A\|_F^2+\|\Delta B\|_F^2}$ 
    \begin{align*}
        \|d(\varphi_E)_X[\Delta]\|_F&\leq 
        \biggl(\sqrt{\|\Delta A\|^2_F+\|\Delta B\|_F^2}+2
        \biggr) \leq 3\\
    \end{align*}
    completing the proof. 
\end{proof}

Applying the mean-value inequality to two local coordinate matrices $X,Y\in \Skew(p)\times \R^{(n-p)\times p}$ yields
\begin{equation}\label{eq:varphi_local_Lipschitz}
    \|\varphi_E(X)-\varphi_E(Y)\|_F\leq 3\cdot \|X-Y\|_F.
\end{equation}
To relate the Frobenius norm $\|\varphi_E(X)-\varphi_E(Y)\|_F$ to the Riemannian distance under the Euclidean metric $\dist_E(\varphi_E(X),\varphi_E(Y))$ we can use that the maximal Fr\'enet curvature of the Stiefel manifold, when endowed with the Euclidean metric, is $\hat \kappa=1$ \cite[p. 305]{zimmermann2025injectivity}, which together with \cite[Property I]{Attali2007} yields the bound
\begin{equation}\label{eq:dist_and_F_norm_comparison}
        \dist_E(\varphi_E(X),\varphi_E(Y))\leq 2\arcsin\qty(\frac{\|\varphi_E(X)-\varphi_E(Y)\|_F}{2}).
    \end{equation}
Alternatively, assuming $n>2p$, we can apply \cite[Theorem 7.1]{mataigne2026bounds} to obtain the same result. The bound of \cite[Theorem 7.1]{mataigne2026bounds} does not assume knowledge of the maximal Fr\'enet curvature, and it can be applied whenever $\St(n,p)$ is endowed with any $\beta$-metric.

We can now state the main result of this section. 
\begin{theorem}\label{thm:main_theorem}
    Let $X, Y\in\Skew(p)\times \R^{(n-p)\times p}$. If $\|X-Y\|_F<\tfrac{2}{3}$ it holds that 
    \begin{equation*}
        \dist_E(\varphi_E(X),\varphi_E(Y))\leq  2\arcsin(\tfrac{3}{2}\|X-Y\|_F).
    \end{equation*}
\end{theorem}
\begin{proof}
    The result follows by combining \eqref{eq:varphi_local_Lipschitz} and \eqref{eq:dist_and_F_norm_comparison}.
\end{proof}
We are now in a position where we can discuss interpolation errors. Given data $U^{(i)}=F(t_i)$ of an unknown (differentiable) function $F$ with $h=\max_i|t_{i+1}-t_i|$, map the data to a neighborhood around $E$ by constructing and applying a suitable isometric group action $\Phi_{(Q,S)}$ according to \Cref{def:Phi}, and map to their local coordinate matrices $X^{(i)}\in \Skew(p)\times \R^{(n-p)\times p}$ via \eqref{eq:psi_Q}. Let $\gamma:I\to \Skew(p)\times \R^{(n-p)\times p}$ be an interpolant such that $X^{(i)}=\gamma(t_i)$ for all $i$, constructed according to any Euclidean interpolation scheme in the tangent space. The corresponding manifold interpolant is 
\begin{equation}\label{eq:manifold_interpol_PL}
    \tilde F(t)=\Phi^{-1}_{(Q,S)}(\varphi_E(\gamma(t))),
\end{equation}
which indeed satisfies $\tilde F(t_i)=F(t_i)$.  
\begin{corollary}
    Let $F(t^*)=U$ and $\tilde F(t^*)=\tilde U$ be respectively true and interpolated data on $\St(n,p)$ situated in a common neighborhood, with local coordinate matrices $X$ and $\tilde X$. Then the interpolation error on $\St(n,p)$ is bounded by 
    \begin{align*}
        \dist_E(F(t^*),\tilde F(t^*))&\leq 2\arcsin((\tfrac{3}{2})\|X-\tilde X\|_F)\\
        &=3\|X-\tilde X\|_F+\mathcal{O}(\|X-\tilde X\|_F^3).
    \end{align*}
\end{corollary}
\begin{proof}
    This is a direct consequence of \Cref{thm:main_theorem} and the Taylor series of $\arcsin$.
\end{proof}
The above result shows that when passing from local coordinates to the manifold,
the interpolation order is preserved and the constant bounds are at most amplified by a factor of $3$. For example, let $U^{(1)}=U(t_1)$ and $U^{(2)}=U(t_2)$ be data in a common open neighborhood, and assume that $t_2-t_1=h\to 0$. We may transfer the classic Euclidean interpolation error bounds of e.g. Lagrange and Hermite interpolation \cite[Sections 8.1 - 8.5]{Alfio:2007} to obtain manifold interpolation error bounds for $t^*\in (t_1,t_2)$ and $K_L,K_H>0$
\begin{align*}
    \dist_E(U(t^*),\tilde U_L(t^*))&\leq  3K_Lh^2=\mathcal{O}(h^2)\\
    \dist_E(U(t^*),\tilde U_H(t^*))&\leq 3K_Hh^4=\mathcal{O}(h^4),
\end{align*}

where $\tilde U_L:(t_1,t_2)\to \St(n,p)$ and $\tilde U_H:(t_1,t_2)\to \St(n,p)$ are the Lagrange and Hermite interpolants, respectively.

\subsection{The Cayley setting} 
As a rule, it is numerically beneficial in terms of computational effort and numerical stability to approximate the matrix exponential and the matrix logarithm by the Cayley transformations \eqref{eq:Cay_trafo}, \eqref{eq:Cay_trafo_inv}.
Applied to the polar-light retraction, the resulting coordinate chart and parametrization around $E$ are 
\begin{equation}\label{eq:coordinate_cay_chart}
\psi_E^{\textsf{Cay}}(U)=
\psi_E^{\textsf{Cay}}\qty(\begin{bmatrix} U_1\\ U_2\end{bmatrix})=
\begin{bmatrix}
	2\Cay^{-1}\left(U_1(U_1^TU_1)^{-\frac12}\right)\\
	U_2(U_1^TU_1)^{-\frac12}
\end{bmatrix}
=:
\begin{bmatrix}
	A\\
	B
\end{bmatrix}
\end{equation}
\begin{equation}\label{eq:param_cay}
    \varphi^{\textsf{Cay}}_E\left(\begin{bmatrix}
        A\\
        B
    \end{bmatrix}\right)=\begin{bmatrix}
        \Cay(\tfrac{1}{2}A)\\
        B
    \end{bmatrix}(I_p+B^TB)^{-\frac12},
\end{equation}
which can be modified so that they are centered around any point $\hat U\in \St(n,p)$ in the same way as for \eqref{eq:psi_Q} and \eqref{eq:varphi_Q}.

$\Cay(A)$ can be computed for any $A\in \Skew(p)$, since $A$ has purely imaginary eigenvalues, which excludes an eigenvalue of $-1$. Letting $H\in \Skew(p)$ we obtain the directional derivatives
\begin{equation}\label{eq:derivative_Cayley}
    d(\Cay)_A[H]=2(I-A)^{-1}H(I-A)^{-1}.
\end{equation}
Similarly, the directional derivative of $\Cay^{-1}$ is
\begin{equation}\label{eq:derivative_Cayley_inv}
    d(\Cay^{-1})_R[P]=2(R+I)^{-1}P(R+I)^{-1}.
\end{equation}
\begin{lemma}\label{lem:cond_cayinv}
    The absolute condition of $\Cay^{-1}:O(p)\to \Skew(p)$ at $R\in O(n)$ satisfies
    \begin{equation*}
        \|d(\Cay^{-1})_R\|_F\leq\frac{2}{(\min_\mu |1+\mu|)^2} 
    \end{equation*}
    where $-1\neq \mu\in S^1\subset \C$ are the eigenvalue of $R$ on the complex unit circle $S^1=\{e^{i\alpha}\mid \alpha\in \R\}$.
\end{lemma}
\begin{proof}
    The bound is obtained in a similar way as in the proof of \Cref{lem:condition_invSPD_factor}.
\end{proof}
The bound in \Cref{lem:cond_cayinv} attains its minimum whenever all eigenvalues are equal to $1$, i.~e., at $R=I$. If we replace $\log_m$ with $2\Cay^{-1}$ in \Cref{eq:psi_Q}, the discussion of applying the group action in \Cref{def:Phi} in \Cref{sec:to_loc_coords} applies equally here. We now consider the conditioning of the mapping in the reverse direction.  
\begin{lemma}\label{lem:cond_cay_trans}
    The Cayley transform $\Cay:\Skew(p)\to O(p)$ has the absolute condition in the induced Frobenious norm \eqref{eq:ind_F_operator_norm}
    \begin{equation*}
        \|d(\Cay)_A\|_F\leq 2\|(I-A)^{-1}\|_2^2\leq 2 \quad \forall A\in\Skew(p). 
    \end{equation*}
\end{lemma}
\begin{proof}
    Take $H\in\Skew(p)$ with $\|H\|_F=1$. From the derivative of $\Cay(A)$ in  \eqref{eq:derivative_Cayley}, 
    \begin{equation*}
        \|d(\Cay)_A[H]\|_F=2\|(I-A)^{-1}H(I-A)^{-1}\|_F\leq 2\|(I-A)^{-1}\|^2_2\|H\|_F,
    \end{equation*}
    and so $\|d(\Cay)_A\|_F\leq 2\|(I-A)^{-1}\|^2_2$. 

    Using the real Schur form $A=PTP^T$, where $T$ is block-diagonal, it follows that 
    \begin{equation*}
        \|(I-A)^{-1}\|_2=\|(I-T)^{-1}\|_2.
    \end{equation*}
    Each ($2\times 2$)-block $[(I-T)^{-1}]_{ii}$ of $(I-T)^{-1}$ is given by
    \begin{equation*}
        [(I-T)^{-1}]_{ii}=\left(\begin{bmatrix}
            1&\theta_i\\
            -\theta_i&1
        \end{bmatrix}\right)^{-1}=\frac{1}{1+\theta_i^2}\begin{bmatrix}
            1&-\theta_i\\
            \theta_i&1
        \end{bmatrix}
    \end{equation*}
    and so $\|[(I-T)^{-1}]_{ii}\|_2=\frac{\sqrt{1+\theta_i^2}}{1+\theta_i^2}\leq 1$. The ($1\times 1$) blocks of $(1-T)^{-1}$ are all equal to $1$, and so $\|(I-A)^{-1}\|_2\leq 1$, completing the proof. 
\end{proof}
We now state the counterpart to \Cref{lem:cond_phi} with $\varphi_E$ replaced by $\varphi_E^{\textsf{Cay}}$, 
\begin{lemma}\label{lem:cond_phi_Cay}
    The absolute condition of $\varphi_E^{\textsf{Cay}}$ at  $X\in \Skew(p)\times \R^{(n-p)\times p}$ in \eqref{eq:param_cay} is bounded
    \begin{equation*}
        \|d (\varphi^{\textsf{Cay}}_E)_X\|_F\leq 3. 
    \end{equation*}
\end{lemma}
\begin{proof}
    The proof is essentially the same as that of \Cref{lem:cond_phi}, where \Cref{lem:cond_cay_trans} implies the inequality $\|d(\Cay)_{(\tfrac{A}{2})}[\Delta A]\|_F^2\leq \|\Delta A\|_F^2$.
\end{proof}
We can now formulate the Cayley version of \Cref{thm:main_theorem} in the setting of the coordinate chart and parameterization \eqref{eq:coordinate_cay_chart} and \eqref{eq:param_cay}.
\begin{theorem}\label{thm:main_theorem_Cayley}
    Let $X, Y\in\Skew(p)\times \R^{(n-p)\times p}$. If $\|X-Y\|_F<\tfrac{2}{3}$ it holds that 
    \begin{equation*}
        \dist_E(\varphi^{\textsf{Cay}}_E(X),\varphi^{\textsf{Cay}}_E(Y))\leq  2\arcsin(\tfrac{3}{2}\|X-Y\|_F).
    \end{equation*}
\end{theorem}


\subsection{Conditioning of the Cayley retraction}

The only second-order accurate retraction under the canonical metric with closed-form inverse known to us is the Cayley retraction from \eqref{eq:cayley_retraction}. (It is based on, but must not be confused with the Cayley transformation.) 
We restate the expressions 
\begin{align}
    \mcR_{\hat U}(\xi)&=\Cay\qty(\tfrac{1}{2}(P_{\hat U}\xi {\hat U}^T-{\hat U}\xi^TP_{\hat U})){\hat U}.
    \tag{\ref{eq:cayley_retraction}}\\
    \mcR_{\hat U}^{-1}(U)&=2{\hat U}F(U)^T+2UF(U)-2{\hat U}, \ F(U)=(I+{\hat U}^TU)^{-1}. 
    \tag{\ref{eq:inv_Cay_ret}}
\end{align}
When using \eqref{eq:cayley_retraction}, the conditioning of mapping data from $\St(n,p)$ to a local coordinate neighborhood in $T_U\St(n,p)$ is essentially governed by the condition number of $F(V)=(I_p+\hat U^TV)^{-1}$. Differentiating $\mcR_{\hat U}^{-1}(V)$ yields, with $\Delta\in T_V\St(n,p)$,
\begin{equation}\label{eq:dinv_Cay_ret}
    d(\mcR_{\hat U}^{-1})_V[\Delta]=-2\hat UF(V)^T (\Delta^T\hat U) F(V)^T-2VF(V) (\hat U^T\Delta) F(V)+2\Delta F(V).
\end{equation}
Taking norms we obtain for $\|\Delta\|_2=1$
\begin{equation}
\label{eq:cond_invcanon_cayley}
    \|d(\mcR_{\hat U}^{-1})_V[\Delta]\|_2\leq 2(2\|F(V)\|^2_2+\|F(V)\|_2).
\end{equation}
The bound is large whenever $I+\hat U^TV$ is close to being singular, or equivalently whenever the matrix $\hat U^TV$ has an eigenvalue close to $-1$. See also the upcoming \Cref{sec:well_defined_inverse}.

\begin{figure}[!t]
    \centering
    \includegraphics[width=0.6\linewidth]{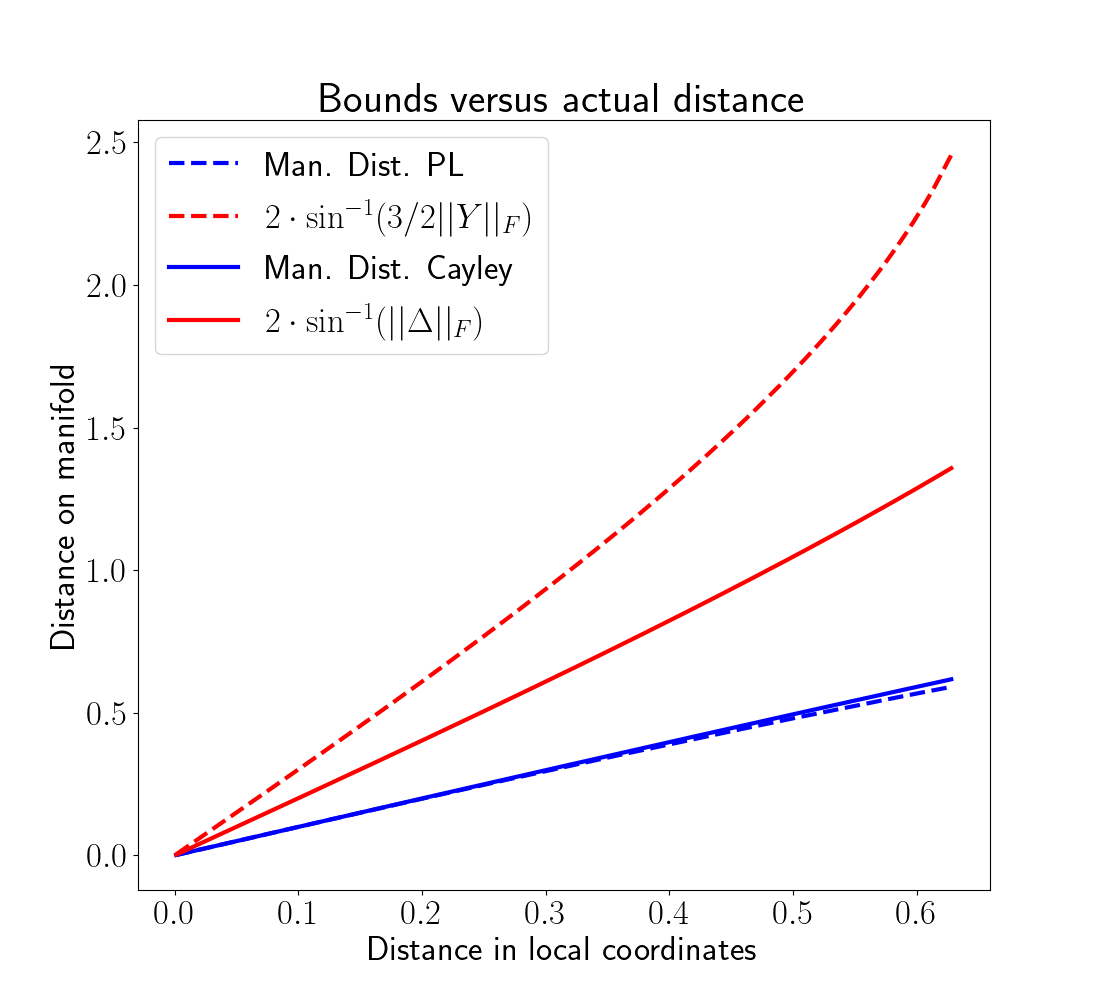}
    \caption{Illustration of the error bounds of \Cref{thm:main_theorem} and \Cref{thm:main_theorem_Cay_ret} on $\St(500,120)$. We generated 200 random points in a common neighborhood of $\psi_E(E)=0$, $\qty{Y^{(i)}}_{i=1}^{200}\subset \Skew(p)\times \R^{(n-p)\times p}, Y^{(i)}=(
        A^{(i)}, B^{(i)})$, and map them to the manifold using $\varphi_E$ in \eqref{eq:varphi_Q} or transform each $Y^{(i)}$ into a tangent vector by the identification $Y^{(i)}\simeq \begin{bmatrix}
            A^{(i)} \\ B^{(i)}
        \end{bmatrix}=\Delta^{(i)}\in T_E\St(n,p)$ and map to the manifold using $\mcR_E$ in \eqref{eq:cayley_retraction}. The manifold distance is computed under the Euclidean metric as $\dist(E,\varphi_E(Y^{(i)}))=\|\Log_{E}(\varphi_E(Y^{(i)}))\|_F$ (blue, dashed) and $\dist_E(E,\mcR_E(\Delta^{(i)}))=\|\Log_{E}(\mcR_E(\Delta^{(i)}))\|_F$ (blue, solid), and the associated bounds of \Cref{thm:main_theorem} (red, dashed) and \Cref{thm:main_theorem_Cay_ret} are shown for comparison.
    }
    \label{fig:illustrate_err_bounds}
\end{figure}
For the direction `tangent space to manifold' we consider the conditioning of the retraction. 
Let $\Delta\in T_{R_{\hat U}(\xi)}\St(n,p), \|\Delta\|_F=1$.
With inner function $h(\xi)= \tfrac{1}{2}(P_{\hat U}\xi {\hat U}^T-{\hat U}\xi^TP_{\hat U})$ we obtain by the chain rule and submultiplicativity
\begin{equation*}
    \mathcal{R}_U(\xi) = \Cay(h(\xi))U \Rightarrow \|d(\Cay\circ h)_{\xi}[\Delta]\|_F \leq  \|d\Cay_{h(\xi)}\|_F\|dh_{\xi}[\Delta]\|_F.
\end{equation*}
As $\|dh_{\xi}[\Delta]\|_F= \frac12 \|P_{\hat U}\Delta \hat U ^T - \hat U\Delta^TP_{\hat U}\|_F\leq\|\hat U\Delta^TP\|_F=\|\Delta^TP_{\hat U}\|_F\leq \|\Delta\|_F\|P_{\hat U}\|_2$, it follows by $\|P_{\hat U}\|_2=1$  that $\|dh_{\xi}[\Delta]\|_F\leq 1$. Using \Cref{lem:cond_cay_trans} we obtain 
\begin{equation}
    \|d(\Cay\circ h)_{\xi}\|_F \leq 2.
\end{equation}
Hence, we have
\begin{lemma}
    \label{lem:cond_canon_cayley}
    The absolute condition of $\mathcal{R}_{\hat U}$ at  $\xi\in T_{\hat U}St(n,p)$ in \eqref{eq:cayley_retraction} is bounded
    \begin{equation*}
        \|d (\mathcal{R}_{\hat U})_{\xi}\|_F \leq 2. 
    \end{equation*}
\end{lemma}
By the mean-value inequality 
\begin{equation*}
    \|\mcR_U(\xi)-\mcR_U(\eta)\|_F\leq 2\|\xi-\eta\|_F,
\end{equation*}
and so we have the following result.
\begin{theorem}\label{thm:main_theorem_Cay_ret}
    Let $\xi,\eta \in T_U\St(n,p)$. If $\|\xi-\eta\|_F<1$ it holds that 
    \begin{equation*}
        \dist_E(\mcR_U(\xi),\mcR_U(\eta))\leq 2\arcsin(\|\xi-\eta\|_F),
    \end{equation*}
    where $\dist_E$ is the Riemannian distance under the Euclidean metric. 
\end{theorem}

The error bounds of \Cref{thm:main_theorem} and \Cref{thm:main_theorem_Cay_ret} are presented in \Cref{fig:illustrate_err_bounds}. 
We observe that the true manifold distances and the error bounds differ as the distance in local coordinates increases.

\subsection{Guaranteeing well-defined inverse retractions}\label{sec:well_defined_inverse}
Restricted to a neighborhood of $E$, when computing the inverse canonical Cayley retraction \eqref{eq:inv_Cay_ret}, the matrix $(I_p+U^TE)=I_p+U_1^T$ has to be inverted, and so $U_1$ cannot have the eigenvalue $-1$. 
For the polar-light retraction in both the standard form \eqref{eq:pl_retraction} and the Cayley-variant \eqref{eq:coordinate_cay_chart}, it is the polar factor of $U_1$ which cannot have the eigenvalue $-1$. 

To gain geometric insight, consider as an example $E\in \St(4,2)$, and let $U$ be connected to $E$ by a curve $y(t)$, where 
\begin{equation*}
    U=\begin{bmatrix}
    -1&0\\
    0&1\\
    0&0\\
    0&0
\end{bmatrix},  \ \ y(t)=\begin{bmatrix}
        \cos(t)&0\\
        0&1\\
        \sin(t)&0\\
        0&0
    \end{bmatrix}.
\end{equation*}
The curve $t\mapsto y(t)$ comes from computing the Riemannian exponential map under the Euclidean metric in the direction $t\xi=tE_\perp \begin{bmatrix}
    1&0\\
    0&0
\end{bmatrix}$ \cite[pp. 8]{zimmermann2025injectivity}. At $t=\pi$ the curve reaches $y(\pi)=U$. Alternatively, we could have taken $-t\xi$ in the Riemannian exponential and have obtained a similar curve, reflecting that $U$ is not in an open geodesic ball of $E$.  The occurrence of the eigenvalue $-1$ of $U_1$ indicates that $U$ is too far away from the reference point $E$. The first column of $U$ and of $E$ span the same subspace, but are of opposite sign.
This information can also be captured by the subspace angle. The following result shows how to construct an explicit normal neighborhood, in which the canonical Cayley retraction is guaranteed to stay invertible. 
\begin{theorem}\label{thm:no_minus_one_eigs}
    Let $B_{\frac{\pi}4,2}(0)=\qty{\Delta\in T_E\St(n,p):\|\Delta\|_2<\tfrac{\pi}{4}}$. For any $t\in [0,1]$ and $\xi\in B_{\frac{\pi}4,2}(0)$ it holds that $U=\Exp_E^{\textnormal{Canon.}}(t\xi)$ has a nonsingular upper ($p\times p$) block with no eigenvalue equal to $-1$, where $\Exp_E^{\textnormal{Canon.}}$ is the Riemannian exponential map under the canonical metric. 
\end{theorem}
\begin{proof}
    As $\xi=\begin{bmatrix}
        A\\
        B
    \end{bmatrix}\in B_{\frac{\pi}4,2}(0)$ it holds that $M=\begin{bmatrix}
        A&-B^T\\
        B&0
    \end{bmatrix}$ has norm $\|M\|_2<\tfrac{\pi}{2}$. Applying the complex Schur decomposition yields 
    \begin{equation*}
        M=P\begin{bmatrix}
            i\lambda_1& & \\
            & \ddots &\\
            & & i\lambda_k
        \end{bmatrix}P^*,
    \end{equation*}
    where  $\lambda_{j}\in (-\frac{\pi}{2},\frac{\pi}{2})$ and $P^*P=I_n$. It follows that $\exp_m(M)=PDP^*$ with $D=\diag(e^{i\lambda_1},\dots,e^{i\lambda_k})$ and so $U=PDP^*E$. The upper $(p\times p)$ block is selected by pre-multiplying by $E^T=E^*$, so we can consider the matrix $E^*PDP^*E$, where we note that $P^*E:=V$ has full row rank since $P$ is unitary. Let $0\neq x\in \C^p$. Then 
    \begin{equation*}
        x^*V^*DVx=(Vx)^*D(Vx):=y^*Dy=\sum_{j=1}^p e^{i\lambda_j}|y_i|^2
    \end{equation*}
    Since $\lambda_j\in (-\frac{\pi}{2},\frac{\pi}{2})$ we have $\cos(\lambda_j) > 0$, and so the real part of the sum is strictly positive. This implies that any nontrivial eigenvalue of $E^*PDP^*E$ has strictly positive real part\footnote{Note that this does not imply that $V^*DV$ is positive definite.}, and it follows that $E^*PDP^*E$ is nonsingular for $\xi\in B_{\frac{\pi}4,2}(0)$. Scaling $M$ with $t\in[0,1]$ implies $t\lambda_j\in (-\frac{\pi}{2},\frac{\pi}{2})$, proves the claim. 
\end{proof}
The result generalizes to all of $\St(n,p)$. Let $\hat B_{\frac{\pi}4,2}(0)=\qty{\Delta\in T_{\hat U}\St(n,p): \|\Delta\|_2<\tfrac{\pi}{4}}$. Then for $\xi\in \hat B_{\frac{\pi}4,2}(0)$, where $\xi=\begin{bmatrix}
    \hat U&\hat U_\perp
\end{bmatrix}\begin{bmatrix}
    A\\
    B
\end{bmatrix}$,
\begin{equation}\label{eq:generalize_to_all_St}
    \begin{array}{cl}
    &V=\Exp_{\hat U}^{\textnormal{Canon}}(\xi)=\begin{bmatrix}
    \hat U& \hat U_\perp
\end{bmatrix}\exp_m\qty(\begin{bmatrix}
        A&-B^T\\
        B&0
    \end{bmatrix})E\\
    \Rightarrow & 
    \hat U^TV
    =E^T\exp_m\qty(\begin{bmatrix}
        A&-B^T\\
        B&0
    \end{bmatrix})E,
    \end{array}
\end{equation}
and the same argument as in the proof above applies. The result provides a quantitative indication of the locality of the local coordinate chart induced by the canonical Cayley retraction.

For the polar-light coordinates, it is the polar factor of $U_1$, $R=U_1(U_1^TU_1)^{-\frac{1}{2}}$ which cannot have the eigenvalue $-1$. If $R$ has the eigenvalue $-1$, then $Rx=-x$ for some $x\in \R^p$ and so $(R-I)x=-2x\Rightarrow \|R-I\|_2\geq 2$, and as $\|R-I\|_2\leq 2$ we have $\|R-I\|_2=2$. The task is therefore to construct a normal neighborhood $\mathcal{B}_E$ of $E$ for which the polar factors of the $U_1$-block in $U\in \mathcal{B}_E$ satisfy $\|R-I\|_2<2$.   We state the following result which is a specialized corollary of \cite[Theorem 2.3]{Roy:1993} 
\begin{lemma}\label{lem:special_corollary_of_23}
    Let $U_1=\tilde R\tilde H$ be a real perturbation of $I_p$ with offset $\Delta U_1 = U_1-I_p$ and polar factorization $\tilde R\tilde H$, and assume $\|\Delta U_1\|_2=\|U_1-I_p\|_2<1$ . Then,
    \begin{equation*}
        \|\tilde R-I_p\|_2\leq 2\|\Delta U_1\|_2.
    \end{equation*}
    As a consequence, if $\|\Delta U_1\|_2 <1$, then the polar factor $\tilde R$ associated with $U_1$ cannot have an eigenvalue of $-1$.
\end{lemma}

\begin{proof}
    Note that the polar factor of $I_p$ is $I_p$ itself and that $\|\Delta U_1\|_2 = \sigma_1(\Delta U_1) < 1 = \sigma_p(I_p)$. Hence, the prerequisites of \cite[Theorem 2.3]{Roy:1993} are fulfilled. Applying the theorem for the 2-norm in the situation at hand gives
    \[
    \|\tilde R - I_p\|_2\leq
     -2\log\left(1-
    \frac{\|\Delta U_1\|_2}{2}\right) \leq 2\|\Delta U_1\|_2.
    \]
    Here, we used that $\log(x)\geq \frac{x-1}{x}, x>0$, which for $x=1-\frac{t}{2}$ yields $2\log(1-\frac{t}{2})\geq 2\frac{-\frac{t}{2}}{1-\frac{t}{2}}\Leftrightarrow -2\log(1-\frac{t}{2})\leq \frac{t}{1-\frac{t}{2}}\leq 2t$ whenever $0\leq t\leq 1$.
\end{proof}
Let $U$ be a perturbation of $E$, written in normal coordinates as $U=\Exp_E^{\textnormal{Canon}}(\xi)$. Then 
\begin{equation*}
    U=E+(U-E):=E+\Delta U, \ \ \Delta U=\begin{bmatrix}
        \Delta U_1\\
        \Delta U_2
    \end{bmatrix}.
\end{equation*}
From the considerations above \Cref{lem:special_corollary_of_23}, for the polar factorization $U_1=\tilde R\tilde H$, we can ensure that $\|\tilde R-I_p\|<2$ by ensuring $\|\Delta U_1\|_2=\|E^T\Exp_E(\xi)-I_p\|_2<1$. The task is to determine $B_{\varepsilon,2}(0)$ such that  for $\xi\in B_{\varepsilon,2}(0)\subseteq T_E\St(n,p)$ the above inequality is valid.
For any $\xi=\begin{bmatrix}
    A\\
    B
\end{bmatrix}\in T_E\St(n,p)$, let $M=\begin{bmatrix}
        A&-B^T\\
        B&0
    \end{bmatrix}\in\Skew(n)$.
\begin{align*}
    \norm{E^T\Exp_E^{\textnormal{Canon}}(\xi)-I_p}_2&=\norm{E^T\exp_m\qty(M)E-E^TE}_2=\norm{E^T\qty(\sum_{j=1}^\infty \frac{1}{j!}M^j )E}_2\\
    &\leq \norm{\exp_m\qty(M)-I_n}_2
    \leq \norm{P\exp_m(D)P^*-I_n}_2\\ 
    &=\max_j|e^{i\lambda_j}-1| =\max_j\qty|2\sin\qty(\frac{\lambda_j}{2})|,
\end{align*}
where $M=P\diag(i\lambda_1,\dots,i\lambda_k)P^*$ is the Schur form., Hence, if $\lambda_j<\frac{\pi}{3}$ we have $\norm{E^T\Exp_E(\xi)-I_p}_2<1$. 
The norm bound $\norm{M}_2<\frac{\pi}{3}$ is guaranteed to hold
if $\|\xi\|_2<\frac{\pi}{6}$. We have the following result. 
\begin{theorem}\label{thm:no_minus_one_eigs_2}
    Let $B_{\frac{\pi}6,2}(0)=\qty{\Delta\in T_E\St(n,p):\|\Delta\|_2<\frac{\pi}{6}}$
    be the ball of spectral radius $\frac{\pi}{6}$. For any $t\in [0,1]$ and any $\xi\in B_{\frac{\pi}6,2}(0)$, it holds that $U=\Exp_E^{\textnormal{Canon}}(t\xi)$ has a nonsingular upper ($p\times p$) block whose polar factor $\tilde R\in O(p)$ in the polar decomposition $U=\tilde R\tilde H$ does not feature the eigenvalue $-1$.
\end{theorem}
\begin{proof}
    It remains to be shown that $E^T\Exp_E^{\textnormal{Canon}}(t\xi)$ is nonsingular, but since $B_{\frac{\pi}6,2}(0)(0)\subset B_{\frac{\pi}4,2}(0)$ of \Cref{thm:no_minus_one_eigs}, we have that this holds. 
\end{proof}
By construction of the inverse polar-light retraction, \Cref{thm:no_minus_one_eigs_2} generalizes to all of $\St(n,p)$. 

Quantifying the radius of a corresponding domain under the Euclidean metric is difficult as the Riemannian exponential map in that case consists of a product of non-commuting matrix exponentials.
Plus, the practical benefits would be limited and therefore, we do not pursue this question.

\section{Hermite interpolation}\label{sec:Hermite_interpol}

As an application example, we consider Hermite interpolation on $\St(n,p)$. We exploit the fact that the inverse of the polar-light retraction
\eqref{eq:coordinate_cay_chart} 
and the inverse canonical Cayley retraction \eqref{eq:cayley_retraction}, can be computed in closed-form.

Let $t_1<t_2$.
The Hermite interpolant of a function $f:[t_1,t_2]\to \R^n$  on two sample points $y_1=f(t_1), y_2=f(t_2)$ is a linear combination of sample points and derivatives
\begin{equation}\label{eq:hermite_interpol}
    \gamma(t)=a_{00}(t)f(t_1)+a_{10}(t)f(t_1
    )+a_{10}(t)f'(t_1)+a_{11}(t)f'(t_1),
\end{equation}
where the coefficient functions $a_{00},\dots,a_{11}$ are the standard cubic Hermite polynomials and are listed in, e.g.,  \cite[Section 8.5]{Alfio:2007}.

\subsection{Hermite interpolation on $\St(n,p)$ using polar-light coordinates} 

We consider the coordinate chart \eqref{eq:coordinate_cay_chart} and assume, without loss of generality, that all data lie in a suitable neighborhood of $E$. Then, with point and derivative data $U^{(1)},\dot U^{(1)}$ and $U^{(2)},\dot U^{(2)}$ sampled at $t_1,t_2$, we map the point data to their local coordinate matrices $X^{(1)}=\psi_E^{\textsf{Cay}}(U^{(1)})$, $X^{(2)}=\psi_E^{\textsf{Cay}}(U^{(2)})$. Derivative data can be mapped bijectively to local coordinate images via $d(\psi_E)_U:T_U\St(n,p)\to \Skew(p)\times \R^{(n-p)\times p}$, which can be obtained by letting $U=\begin{bmatrix}
    U_1\\
    U_2
\end{bmatrix}\in \St(n,p)$ and $\Delta=\begin{bmatrix}
    \Delta_1\\
    \Delta_2
\end{bmatrix}\in T_U\St(n,p)$ and computing 
\begin{equation}\label{eq:d_psi}
    d(\psi^{\textsf{Cay}}_E)_U[\Delta]=\begin{bmatrix}
        4(L(U_1)+I_p)^{-1}d L_{U_1}[\Delta_1](L(U_1)+I_p)^{-1}\\
        \Delta_2(U_1^TU_1)^{-\frac12}+U_2(dJ_{U_1}[\Delta_1])
    \end{bmatrix},
\end{equation}
where
\begin{align*}
    L(U_1)=U_1J(U_1), \ \ J(U_1)=(U_1^TU_1)^{-\frac12}.
\end{align*} 
$dL_{U_1}[\Delta_1]=\Delta_1J(U_1)+U_1dJ_{U_1}[\Delta_1]$ and $d J_{U_1}[\Delta_1]=-(U_1^TU_1)^{-\frac12}D(U_1^TU_1)^{-\frac12}$, $D$ being the solution of the Lyapunov equation
\begin{equation*}
    (U_1^TU_1)^{\frac12}D+D(U_1^TU_1)^{\frac12}=U_1^T\Delta_1+\Delta_1^TU_1.
\end{equation*}
The solution $D$ exists and is unique since $(U_1^TU_1)^{\frac{1}{2}}\in SPD(p)$. After mapping all data to their respective local coordinate matrices, one may apply \eqref{eq:hermite_interpol} to obtain the manifold interpolant $\tilde F(t)=\varphi_E^{\textsf{Cay}}(\gamma(t))$. 

If the data is not contained in a neighborhood of $E$, select a point in the neighborhood of the data to act as center and construct the group action $\Phi_{(Q,S)}$ of \Cref{def:Phi}. Applying the group action to the Stiefel data and the derivative information moves the data to be in a neighborhood of $E$. After applying the outlined Hermite interpolation procedure above, the interpolant is given by $\tilde F(t)=\Phi_{(Q,S)}^{-1}(\varphi_E^{\textsf{Cay}}(\gamma(t)))$.

\subsection{Hermite interpolation on $\St(n,p)$ using the canonical Cayley retraction} Consider the inverse Cayley retraction \eqref{eq:inv_Cay_ret} and its derivative \eqref{eq:dinv_Cay_ret}. As in the previous section, assume that we are given point and derivative data $U^{(1)},\dot U^{(1)}$ and $U^{(2)},\dot U^{(2)}$ sampled at $t_1<t_2$. Choose either $T_{U^{(1)}}\St(n,p)$ or $T_{U^{(2)}}\St(n,p)$ as the reference tangent space, and map the data to this space using the inverse retraction \eqref{eq:inv_Cay_ret} and its derivative \eqref{eq:dinv_Cay_ret}. Choosing, for example, $T_{U^{(1)}}\St(n,p)$ as reference and applying \eqref{eq:hermite_interpol}, we obtain the interpolant $\tilde F(t)=\mcR_{U^{(1)}}(\gamma(t))$.

\subsection{Numerical example: Interpolating the Q factor in the QR decomposition}\label{sec:practical_example}

Consider the curve $Y:[0,2]\to \R^{n\times p}$
\begin{equation*}
    Y(t)=A_0+0.25A_1t+0.125A_2t^2+0.05A_3t^3,
\end{equation*}
where $A_0,A_1,A_2,A_3\in \R^{n\times p}$ are pseudo-randomly generated matrices with entries drawn uniformly from $[0,1]$. Consider the Stiefel curve
\begin{equation}\label{eq:Q_t_exp}
    Q(t) = \textnormal{qf}(Y(t)),
\end{equation}
where $\textnormal{qf}(Y(t))$ is the $Q$ factor in the unique compact QR decomposition $Q(t)R(t)=Y(t)$. The decomposition can be differentiated by applying \cite[Proposition 2.2]{WalterLehmannLamour2012}. We sample $Q(t)$ and its derivative $\dot Q(t)$ at $t_0=0$ and $t_1=2$ and apply the methodology for Hermite interpolation via the Cayley variant of the polar-light coordinates (PL) and via the canonical Cayley retraction (Cay) outlined in the previous sections. For reference, we also include Hermite interpolation via Riemannian normal coordinates under the Euclidean metric (RN) as discussed in \cite{ZimmermannHermite_2020}. For the experiments, the reference point for constructing the group action $\Phi_{(\hat Q,S)}$ is $Q(t_0)$. For the canonical Cayley retraction and the Riemannian normal coordinates, we use the tangent space $T_{Q(t_0)}\St(n,p)$.

The result of interpolating \eqref{eq:Q_t_exp} is presented in \Cref{fig:experiment}, where the curve is realized on $\St(80,40)$ and on $\St(800,40)$. It is seen that for $n\gg p$ (see plots (b) and (d)) that (PL) outperforms (Cay) and (RN) in terms of the manifold interpolation error. When $n$ approaches $2p=80$ using (Cay) is seen to lead to the smallest relative errors, and (PL) produces the largest. Surprisingly, (RN) does not outperform either of the two methods in this experiment. It is worth noting that, although the injectivity radius of $\St(n,p)$ under the Euclidean metric is $\pi$ \cite{zimmermann2025injectivity} and bounded by $0.913\pi$ under the canonical metric \cite[Corollary 7.2]{absil2024ultimate}. Yet, these are worst-case bounds are are attained only for special geodesics along low-rank tangent directions \cite{stoye:2026}. The interpolation routines still work for the present examples. 
\begin{figure*}[ht!]
    \centering
    \begin{subfigure}[t]{0.49\textwidth}
        \centering
        \includegraphics[width = 0.9\textwidth]{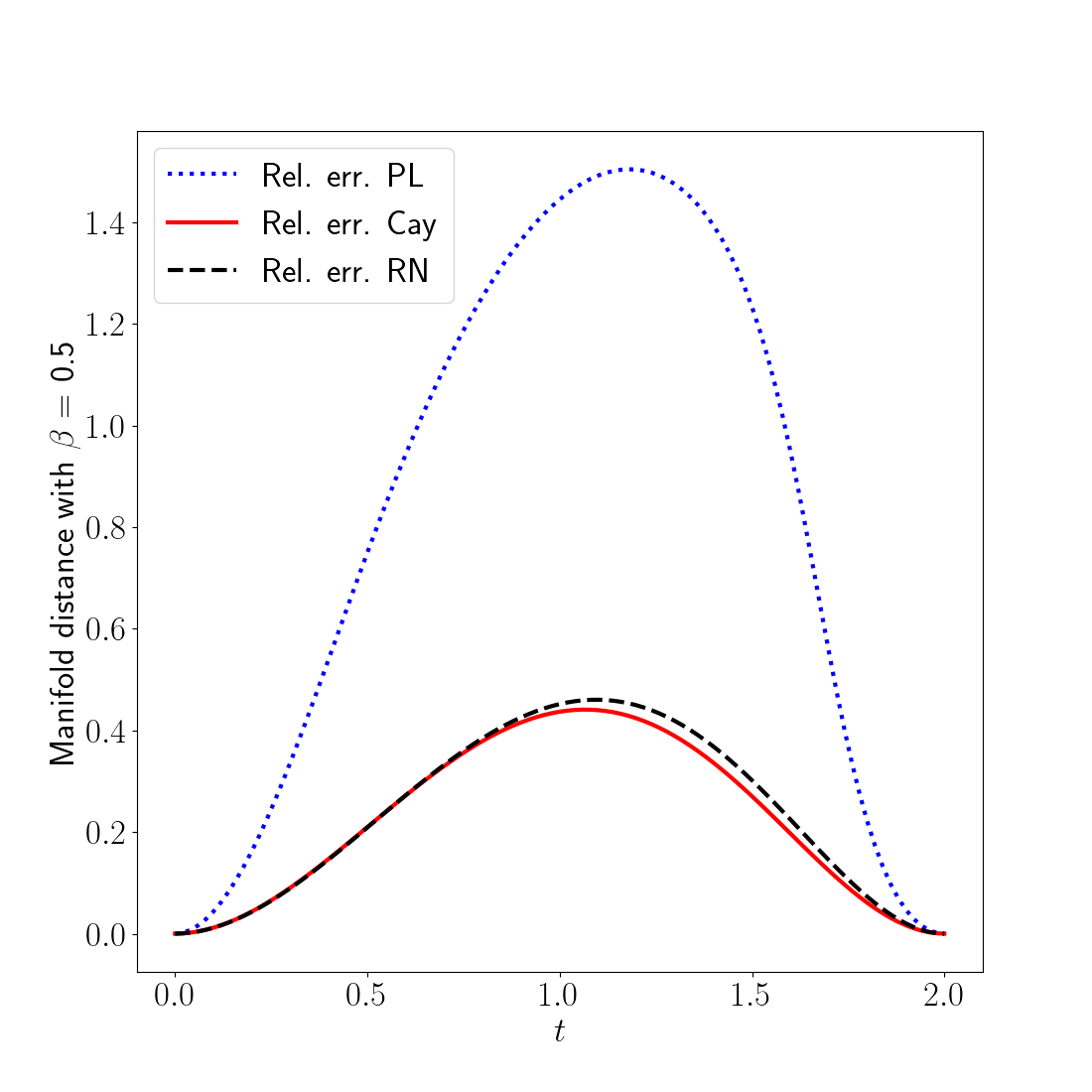}
        \caption{\centering $n=80,p=40$,\\ $\dist_{\textnormal{Canon.}}(U(t_0),U(t_1))=4.76$.}
    \end{subfigure}%
    ~ 
    \begin{subfigure}[t]{0.49\textwidth}
        \centering
        \includegraphics[width = 0.9\textwidth]{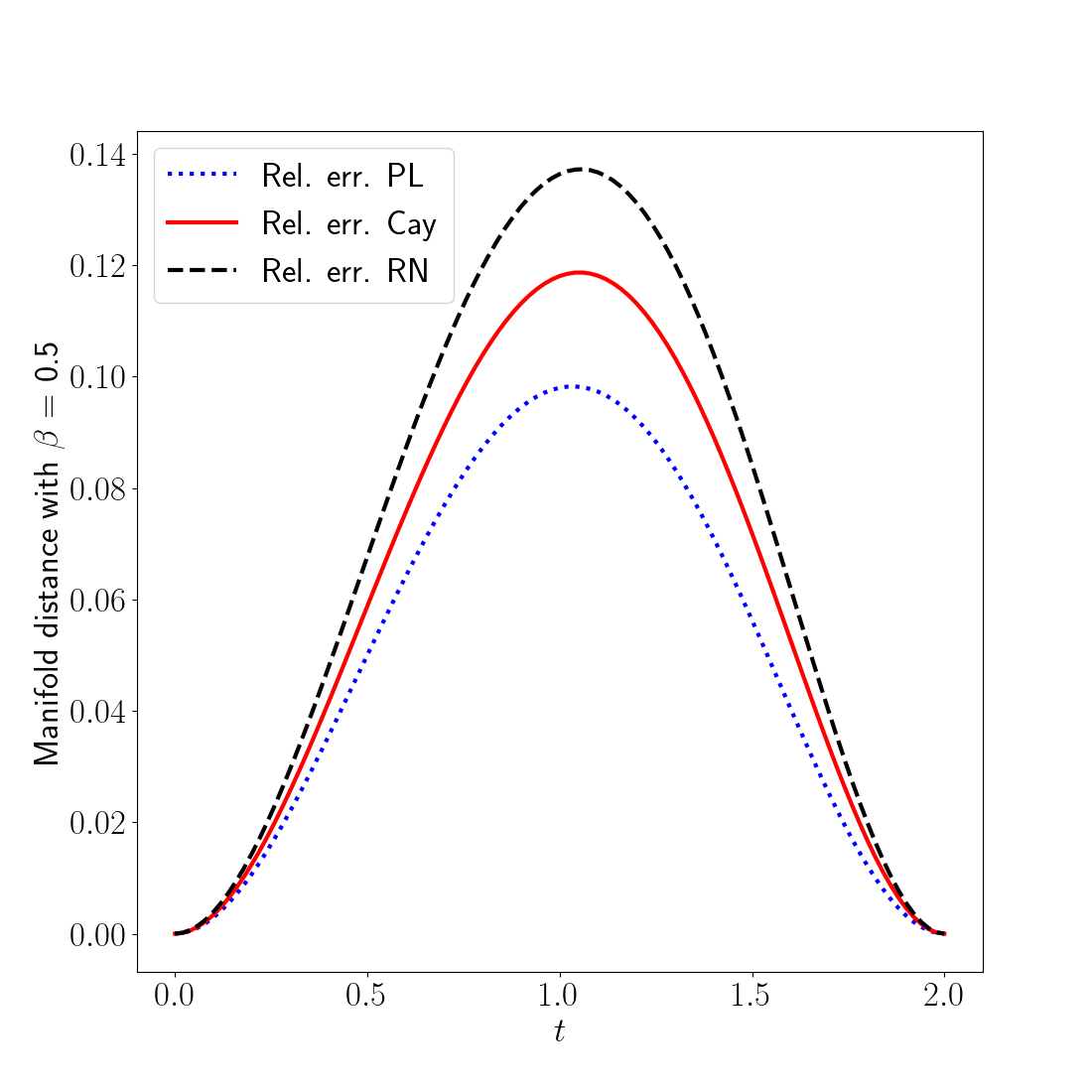}
        \caption{\centering $n=800,p=40$,\\ $\dist_{\textnormal{Canon.}}(U(t_0),U(t_1))=4.30$.}
    \end{subfigure}
    \begin{subfigure}[t]{0.49\textwidth}
        \centering
        \includegraphics[width = 0.9\textwidth]{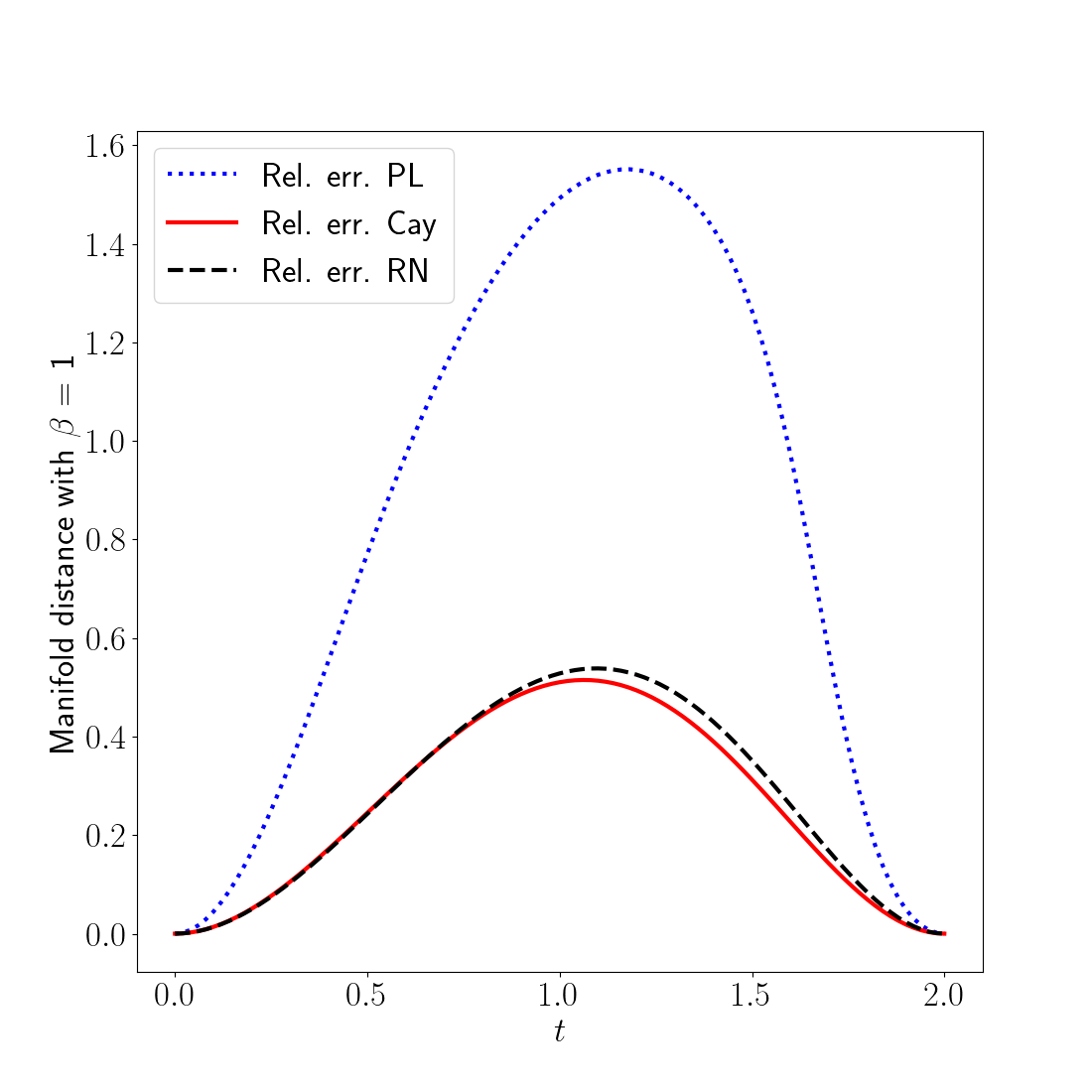}
        \caption{\centering $n=80,p=40$,\\ $\dist_E(U(t_0),U(t_1))=5.30$.}
    \end{subfigure}%
    ~ 
    \begin{subfigure}[t]{0.49\textwidth}
        \centering
        \includegraphics[width = 0.9\textwidth]{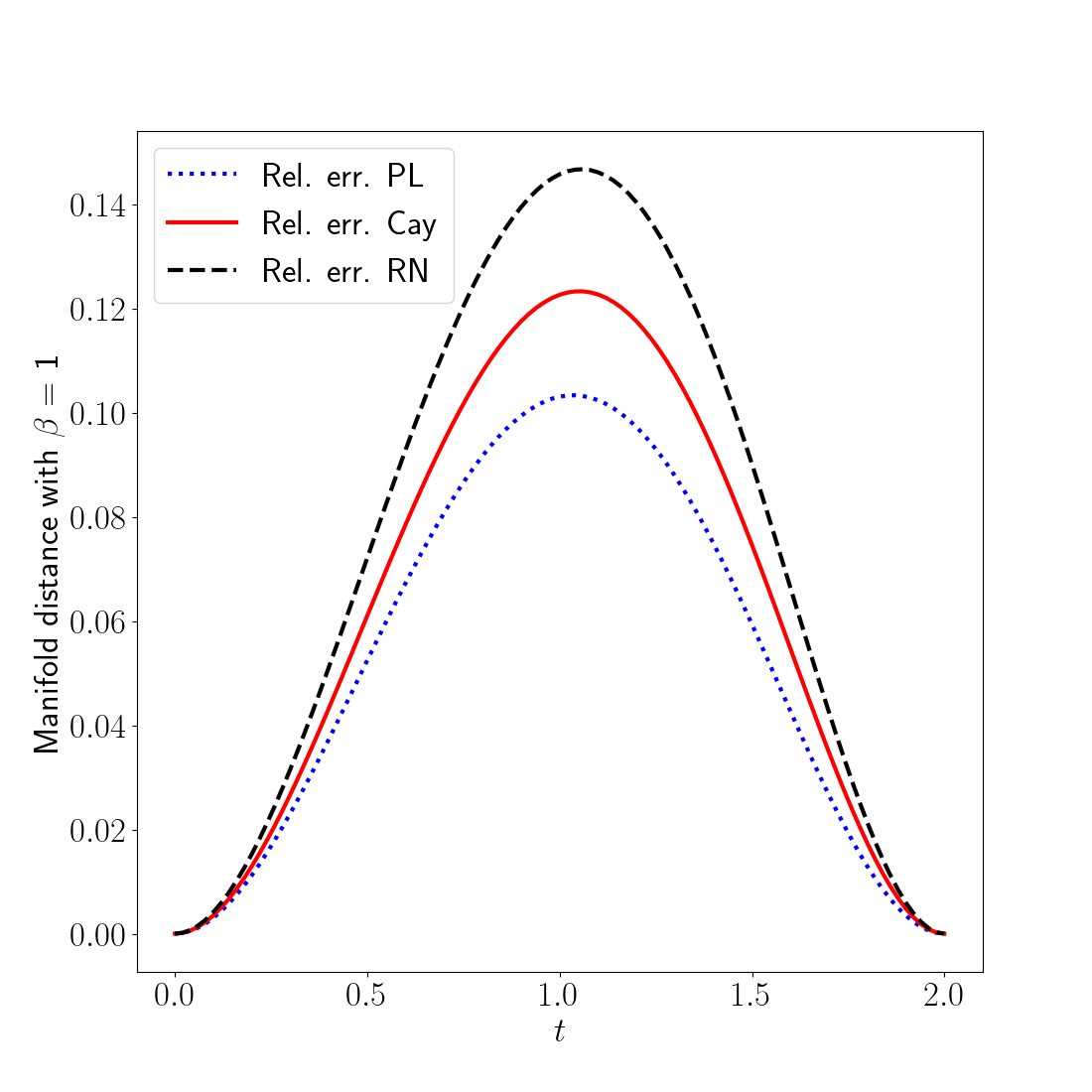}
        \caption{\centering $n=800,p=40$,\\ $\dist_E(U(t_0),U(t_1))=4.35$.}
    \end{subfigure}
    \caption{Manifold interpolation errors of interpolating the Stiefel curve \eqref{eq:Q_t_exp} under two choices of dimensions. $\dist(U(t_0),U(t_1))$ is computed under the canonical metric ($\beta = 0.5$) in the plots (a) and (b) and under the Euclidean metric ($\beta=1$) in plots (c) and (d).}
    \label{fig:experiment}
\end{figure*}

\section{Concluding remarks}\label{sec:concluding_remarks} 
In this article, we have examined the only two known second-order retractions on the Stiefel manifold with explicit inverses,
namely the polar-light retraction which is based on the matrix polar decomposition, and the canonical Cayley retraction
in terms of their conditioning and numerical properties.
In forward-mode, the retractions are well-conditioned.
The absolute condition number of the polar-light retraction is bounded by $3$ (\Cref{lem:cond_phi}, \Cref{lem:cond_phi_Cay}); the absolute condition number of the canonical Cayley retraction is bounded by $2$ (\Cref{lem:cond_canon_cayley}).
The inverses of the retractions at base point $\hat U$ map a Stiefel matrix $U$ to a tangent vector. The inverse maps can become arbitrarily ill-conditioned (\Cref{lem:condition_invSPD_factor}, equation \eqref{eq:cond_invcanon_cayley}). For the inverses of both retractions to be well-defined, it is essential 
\begin{itemize}
    \item that the polar factor $\hat U^T U$ must not have an eigenvalue of $-1$. (polar-light retraction)
    \item that $\hat U^T U$ must not feature an eigenvalue of $-1$.
    (canonical Cayley retraction)
\end{itemize}
If the base point is $\hat U=E=\begin{bmatrix}
    I_p\\
    0
\end{bmatrix}$,
these requirements become conditions for upper diagonal block $U_1$ of $U$. In practice, ill-conditioning can be countered by using a group action to send data to a neighborhood of $E$, which features a perfectly conditioned upper diagonal block. (\Cref{def:Phi}, \Cref{prop:isometry})
We have constructed explicit though conservative normal neighborhoods for any point on $\St(n,p)$, in which we are guaranteed that the inverse canonical Cayley retraction and the inverse polar-light retraction are well-defined. They are given by the images of the open balls $B_{\frac{\pi}{6},2}(0)$ and $B_{\frac{\pi}{4},2}(0)$, respectively, under the Riemannian exponential map under the canonical metric. (\Cref{thm:no_minus_one_eigs}, \Cref{thm:no_minus_one_eigs_2})

As main application, we have considered manifold interpolation. We have shown that the interpolation error, when applying a Euclidean interpolation scheme to interpolate local coordinate images computed by either of the retractions, is propagated to the manifold with asymptocic amplification factor bounded by the condition constants $K=2$ for the canonical Cayley retraction (\Cref{thm:main_theorem_Cay_ret}),
and $K=3$ for the polar-light retraction 
(\Cref{thm:main_theorem}, \Cref{thm:main_theorem_Cayley}).

The numerical example shows that using a retraction can lead to smaller interpolation errors when compared to working in Riemannian normal coordinates, but selecting the optimal coordinates is seemingly problem-dependent. 

\section*{Acknowledgments}

\bibliographystyle{siamplain}
\bibliography{SurrogateModelling}
\end{document}